\documentclass[12pt]{amsart}

\usepackage[latin1]{inputenc}
\usepackage[english]{babel}
\usepackage{amsthm}
\usepackage{amsmath}
\usepackage{amsxtra}
\usepackage{mathrsfs} 
\usepackage{amssymb}
\usepackage[dvips]{graphicx}
\usepackage{graphicx}

\input xy
\xyoption{all}

\newtheorem{teo}{Theorem}[section]
\newtheorem{lemma}[teo]{Lemma}

\newtheorem{defi}[teo]{Definition}
\newtheorem{coro}[teo]{Corollary}

\newtheorem{cla}[teo]{Claim}

\theoremstyle{remark}
\newtheorem{remark}[teo]{Remark}

\begin{document}

\newcommand{\ran}{\texttt{ran}}
\newcommand{\cof}{\texttt{cof}}
\newcommand{\dom}{\texttt{dom}}
\newcommand{\N}{\mathcal{N}}
\newcommand{\up}{\upharpoonright}

\newcommand\Acal{\mathscr{A}}
\newcommand\Bcal{\mathscr{B}}
\newcommand\Ecal{\mathscr{E}}
\newcommand\Fcal{\mathscr{F}}
\newcommand\Hcal{\mathscr{H}}
\newcommand\Ical{\mathscr{I}}
\newcommand\Ncal{\mathscr{N}}
\newcommand\Mcal{\mathscr{M}}
\newcommand\Pcal{\mathscr{P}}
\newcommand\Qcal{\mathscr{Q}}
\newcommand\Rcal{\mathscr{R}}
\newcommand\Tcal{\mathscr{T}}
\newcommand\Ucal{\mathscr{U}}
\newcommand\Zcal{\mathscr{Z}}
\newcommand\Rbb{\mathbb{R}}
\newcommand\Nfrak{\mathfrak{N}}
\newcommand\Pfrak{\mathfrak{P}}
\newcommand\restrict{\restriction}

\newcommand{\diff}{\operatorname{diff}}
\newcommand{\Ht}{\operatorname{height}}
\newcommand{\meet}{\wedge}
\newcommand\triord{\triangleleft}
\newcommand{\Th}{{}^{\mathrm{th}}}
\newcommand{\St}{{}^{\mathrm{st}}}
\newcommand\axiom{\mathrm}
\newcommand\MA{\axiom{MA}}
\newcommand\MM{\axiom{MM}}
\newcommand\PFA{\axiom{PFA}}
\newcommand\MRP{\axiom{MRP}}
\newcommand\SRP{\axiom{SRP}}
\newcommand\ZFC{\axiom{ZFC}}
\newcommand\CAT{\axiom{CAT}}
\newcommand{\<}{\langle}
\renewcommand{\>}{\rangle}
\newcommand\mand{\textrm{ and }}
\renewcommand{\diamond}{\diamondsuit}

\newcommand{\cf}{{\mbox{cof}}}
\newcommand{\height}{\Ht}
\newcommand\NS{\mathrm{NS}}

\title[Preservation of a Souslin tree and side conditions]{Preservation of a Souslin tree and side conditions}
\author{Giorgio Venturi}
\email{gio.venturi@gmail.com}
\address{Centro de 
L\'{o}gica, Epistemologia e Hist\'{o}ria de la Ci\^{e}ncia,  Univ. Est. de Campinas. 
Rua Sérgio Buarque de Holanda, 251
Barão Geraldo, SP.}


\maketitle

\begin{abstract}
We show how to force, with finite conditions, the forcing axiom PFA(T),  a relativization of PFA to proper forcing notions
preserving a given Souslin tree T.  The proof uses a Neeman style iteration with generalized side conditions  consisting of 
models of two types, and a preservation theorem for such iterations. The consistency of this axiom was previously 
known by the standard countable support iteration, using a preservation theorem due to Miyamoto. 
%
\end{abstract}

\section*{Introduction}

In this article, using the techniques introduced by Neeman in \cite{Neeman2}, 
we give a consistency proof of the Forcing Axiom for the class of proper
forcings that preserve a Souslin tree $T$ i.e. PFA$(T)$\footnote{See \cite{Todo},
for a survay of interesting applications of PFA($T$).}.
The novelty of this proof is that PFA($T$) is forced 
with finite conditions, using a forcing that acts like an iteration.
Indeed, the known consistency proofs for this axiom made use 
of a result of Miyamoto (\cite{Miyamoto}), 
who showed that the property ``is proper and preserves every $\omega_1$-Souslin tree''
is preserved under a countable support iteration of proper
forcings.

The main preservation theorem presented here, Theorem \ref{preservation}, can be seen as 
a general preservation schema for properties, like being a Souslin tree, that have
formulations similar to Lemma \ref{charact}, in terms of the possibility to
construct a generic condition for a product forcing, by means of conditions that, singularly, are
generic for their respective forcings. As a matter of fact, in the proof of Theorem 
\ref{preservation}, no use is made of the fact that $T$ is a tree. 

In Section \ref{Souslinprop} we review some basic results connecting the property of 
being Souslin and properness. In Section \ref{Souslincount} we show, as a warm up, 
that the method of side conditions - with just countable models -
does not influence the fact that a proper forcing $\mathbb{P}$ preserves a Souslin tree $T$. 
Then in Section \ref{PFAT} we use the method of generalized side conditions 
with models of two types to construct a model where PFA($T$) holds and $T$ remains Souslin.  
We refer to \cite{Neeman2} and \cite{VelVen2} for a detailed presentation of a
pure side conditions poset with both countable and uncountable models.

\section{Souslin trees and properness}\label{Souslinprop}

We will use the following reformulation of the definition of Souslin tree.

\begin{lemma}\label{Souslin}
A tree $T$ is Souslin iff for every countable $M \prec H(\theta)$, with $\theta$ sufficiently large such that $T \in M$, and for every $t \in T_{\delta_{M}}$, where 
$\delta_M = M \cap \omega_1$,
$$t \mbox{ is an } (M, T)\mbox{-generic condition},$$
i.e. for every maximal antichain $A \subseteq T$ in $M$, there is a $\xi < M\cap \omega_1$ such that $t\up \xi \in A$. 
\end{lemma}
\begin{proof}
On the one hand, let $T$ be a Souslin tree, $M \prec H(\theta)$ as above, $t \in T_{\delta_M}$ and $A \in M$ a maximal antichain of $T$. Since $T$ is Souslin, $A$ is countable. Then there is a $\alpha < \delta_M$ such that for all $\beta \geq \alpha$, the set $A \cap T_\beta$ is empty. Hence there is an element $h \in A$ compatible with $t\up \alpha$. Then $t \up ht(h) = s \in A$.

On the other hand if $A \in M$ is an uncountable maximal antichain of $T$, then $A \setminus M$ is not empty. For $x \in A \setminus M$, let $t = x \up \delta$. If there is a $\xi < \delta$ such that $t \up \xi \in A$, then $x$ and $t \up \xi$ would be compatible and both in $A$: a contradiction. 
\end{proof}

The following lemma connects preservation of Souslin trees and properness. 

\begin{lemma}\label{charact}(Miyamoto, Proposition 1.1 in \cite{Miyamoto})
Fix a Souslin tree $T$, a proper poset $\mathbb{P}$ and some regular cardinal $\theta$, large enough. Then the following are equivalent:
\begin{enumerate}
\item $\Vdash_{\mathbb{P}}$ `` $T$ is Souslin '', 
\item given $M \prec H(\theta)$ countable, containing $\mathbb{P}$ and $T$, if $p \in \mathbb{P}$ is a $(M, \mathbb{P})$-generic condition and 
$t \in T_{\delta_M}$, with $\delta_M = M \cap \omega_1$, then $(p,t)$ is an $(M, \mathbb{P} \times T)$-generic condition, 
\item given $M \prec H(\theta)$ countable, containing $\mathbb{P}$ and $T$ and given $q \in \mathbb{P} \cap M$, there 
is a condition $p \leq q$ such that for every condition $t \in T_{\delta_M}$, with $\delta_M = M \cap \omega_1$, we have that 
$(p,t)$ is an $(M, \mathbb{P} \times T)$-generic condition. 
\end{enumerate}
\end{lemma} \qed

\section{Preservation of $T$ and countable models}\label{Souslincount}

We define the scaffolding operator from an idea of Veli\v{c}kovi\'{c}.

\begin{defi}
Given a proper poset $\mathbb{P}$ and a sufficiently large cardinal $\theta$ such that
$\mathbb{P} \in H(\theta)$, let $\mathbb{M}(\mathbb{P})$ be the poset 
consisting of conditions $p = (\mathcal{M}_p, w_p)$ such that
\begin{enumerate}
\item $\mathcal{M}_p$ is a finite $\in$-chain of countable elementary substructures of $H(\theta)$,
\item $w_p \in \mathbb{P}$,
\item $w_p$ is an $(M, \mathbb{P})$-generic condition for every $M$ in $\mathcal{M}_p$. 
\end{enumerate}
Moreover, we let $q \leq p$ iff $\mathcal{M}_p \subseteq \mathcal{M}_q$ and $w_q \leq_{\mathbb{P}} w_p$. 
\end{defi}

\begin{remark}
Notice that $\mathbb{M}(\mathbb{P})$ does not make reference to the cardinal $\theta$. However this notation
causes no confusion as long as $\theta$ depends on $\mathbb{P}$ and its choice is a standard negligible part
of all arguments involving properness. Then, without any specification, $\theta$ will always denote a cardinal 
that makes possible the definition of $\mathbb{M}(\mathbb{P})$. 
\end{remark}

\begin{remark}
By abuse of notation we will identify an $\in$-chain $\mathcal{M}_p$ and the set of models that compose it. 
\end{remark}

Our aim now is to show that properness is preserved by the scaffolding operator. 

\begin{lemma}\label{cond}
Let $\mathbb{P}$ be a proper poset, $M \prec H(\theta)$ and $p \in \mathbb{M}(\mathbb{P}) \cap M$. Then 
there is a condition $p^M = (\mathcal{M}_{p^M}, w_{p^M}) \in \mathbb{M}(\mathbb{P})$ that is the largest condition extending $p$ and such that
$M \in \mathcal{M}_{p^M}$.  
\end{lemma}
\begin{proof}
First of all notice that since $p \in M$, we have $\mathcal{M}_p \subseteq M$. In particular the largest model in 
$\mathcal{M}_p$ belongs to $M$. So $\mathcal{M}_p \cup \{M\}$ is a finite $\in$-chain of elementary substructures of
$H(\theta)$. Moreover $w_p \in M \cap \mathbb{P}$ and, by properness,  there is a $w_q \leq w_p$ that is $(M,\mathbb{P})$-generic.
Now, since $w_q \leq w_p$ and $w_p$ is $(N, \mathbb{P})$-generic, for every $N \in \mathcal{M}_p$, so is $w_q$. Then we have
that $w_q$ is a generic condition for every model in $\mathcal{M}_p \cup \{M\}$. 
Finally set $\mathcal{M}_{p^M} =\mathcal{M}_p \cup \{M\}$ and $w_{p^M} = w_q$ to see that the conclusion of the lemma holds.  
\end{proof}

\begin{teo}
Let $\mathbb{P}$ be a proper poset. Then $\mathbb{M}(\mathbb{P})$ is proper.
\end{teo}
\begin{proof}
Let $M^*$ be a countable elementary submodel of $H(\theta^*)$, for some $\theta^* > \theta$, where $\theta$ is the corresponding cardinal 
in the definition of $\mathbb{M}(\mathbb{P})$. If $p$ is a condition in $\mathbb{M}(\mathbb{P}) \cap M^*$ we need to find a condition $q \leq p$
that is $(M^*, \mathbb{M}(\mathbb{P}))$-generic. Fix then a dense $D \subseteq \mathbb{M}(\mathbb{P})$ in $M^*$ and let $M = M^* \cap H(\theta)$.
We claim that $p^M = (\mathcal{M}_p \cup \{M\}, w_p^M)$ is an $(M, \mathbb{M}(\mathbb{P}))$-generic condition.

Thanks to Lemma \ref{cond} we have that $p^M$ is a condition. We now prove its genericity. Let $r \leq p^M$ and without loss of generality
assume it to be in $D$. Define
$$
E = \{w_s \in \mathbb{P} : \exists \mathcal{M}_s \mbox{ such that } (\mathcal{M}_s, w_s) \in D \land \mathcal{M}_r \cap M \subseteq \mathcal{M}_s\} 
$$
and notice that $E \in M^*$ and $w_r \in E$. 

The set $E$ may not be dense in $\mathbb{P}$, but
$$
E_0 = \{w_t \in \mathbb{P} : \exists w_s \in E \mbox{ such that } w_t \leq w_s \mbox{ or } \forall w_s \in E (w_t \perp w_s )\}
$$
is a dense subset of $\mathbb{P}$ that belongs to $M^*$. 

Then thanks to the ($M^*, \mathbb{P}$)-genericity of $w_{p}^M$ and the fact that $w_r \leq w_p^M$, we have that there is a condition $w_t \in M^* \cap E_0$
that is compatible with $w_r$. Since $w_r \in E$  there is a condition $w_s \in E$ such that $w_t \leq w_s$. By elementarity can find $w_s$ in $M^*$. Moreover, 
by definition of $E$, there is an $\mathcal{M}_s$ such that 
$(\mathcal{M}_s, w_s) \in D$ and such that $\mathcal{M}_r \cap M \subseteq \mathcal{M}_s$. Again by elementarity we can find $\mathcal{M}_s$ in $M$. 
Hence $(\mathcal{M}_s, w_s) \in D \cap M^*$. 

Finally notice that $w_s$ is compatible with $w_r$, because $w_t$ is so and $w_t \leq w_s$; let $w_a$ be the witness of it, i.e. $w_a \leq w_s, w_r$. Besides $\mathcal{M}_s \subseteq M$ and it extends $\mathcal{M}_r \cap M$, so we have that $\mathcal{M}_a = \mathcal{M}_s \cup  \{M\} \cup \mathcal{M}_r \setminus M$ is a finite $\in$-chain
of elementary submodel of $H(\theta)$. Then, in order to show that $(\mathbb{M}_a, w_a)$ is a condition in $\mathbb{M}(\mathbb{P})$ we need to show that $w_a$ is 
$(N, \mathbb{P})$-generic, for every $N \in \mathcal{M}_a$. But this is true because on one hand $s \in \mathbb{M}(\mathbb{P})$ and so $w_s$ is $(N, \mathbb{P})$-generic
for every $N \in \mathcal{M}_s$ and on the other hand $r \in \mathbb{M}(\mathbb{P})$ and so $w_r$ is $(N, \mathbb{P})$-generic for every $N \in \mathcal{M}_r$. Since $w_a$ extends both $w_s$ and $w_r$, we have that $w_a$ is generic for all the models in $\mathcal{M}_a$. Hence $a$ extends both $s$ and $r$, in $\mathbb{M}(\mathbb{P})$,  and witnesses their compatibility. 
\end{proof}

We now want to show that the scaffolding operation 
does not effect the preservation of a Souslin tree 
$T$. In order to show this fact
we will use the characterization of Lemma \ref{charact}. 

\begin{lemma}
Let $T$ be a Souslin tree and let $\mathbb{P}$ be a proper forcing, such that $\Vdash_{\mathbb{P}} ``T$ is Souslin''. 
Moreover let $M^*$ be a countable elementary submodel of $H(\theta^*)$, for some $\theta^* > \theta$, where $\theta$ is the corresponding cardinal in the definition of $\mathbb{M}(\mathbb{P})$. 
If $p \in \mathbb{M}(\mathbb{P})$, $M = M^* \cap H(\theta) \in \mathcal{M}_p$  and $t \in T_{\delta_M}$, with $\delta_M = M \cap \omega_1$,
then $(p,t)$ is an $(M^*, \mathbb{M}(\mathbb{P}) \times T)$-generic condition. 
\end{lemma}
\begin{proof}
Fix a set $D \subseteq \mathbb{M}(\mathbb{P}) \times T$ dense in $M^*$ and fix a condition $(r, t') \leq (p, t)$, that without loss of generality we can 
assume to be in $D$. Then define  
$$
E = \{(w_q,h) \in \mathbb{P} \times T| \exists \mathcal{M}_q \mbox{ such that } (q, h) \in D \mbox{ and } \mathcal{M}_r \cap M \subseteq \mathcal{M}_q\} 
$$
and notice that $E \in M$ and $(w_r,t') \in E$. Again the set $E$ may not be dense, but the set $\bar E = E^{\leq} \cup E^{\perp}$, where
$$
E^{\leq} = \{(w_s,u) \in \mathbb{P} \times T | \exists (w_q,h) \in E \mbox{ such that } (w_s,u) \leq (w_q,h) \} \mbox{ and }
$$
$$
E^{\perp} = \{(w_s,u) \in \mathbb{P} \times T | \forall (w_q,h) \in E (w_s,u) \perp (w_q,h) \},
$$
is a dense subset of $\mathbb{P} \times T$ that belongs to $M^*$. 

Now, since $M \in \mathcal{M}_r$, the condition $w_r$ is $(M, \mathbb{P})$-generic, by definition of $\mathbb{M}(\mathbb{P})$. Moreover
since $\Vdash_{\mathbb{P}}$ `` $T$ is Souslin '' we have that $(w_r, t')$ is $(M^*, \mathbb{P} \times T)$-generic. Then there is 
a $(w_s, u) \in \bar E \cap M^*$, that is compatible with $(w_r,t')$. This latter fact then implies that $(w_s, u) \in E^{\leq} \cap M^*$
and so there is a condition $(w_q,h) \in E$ such that $(w_s, u) \leq (w_q,h)$. By elementarity we can find $(w_q,h) \in M^*$ and again, 
by elementarity we can assume $q = (\mathcal{M}_q, w_q)$ to be in $ M^*$ and so $(q,h) \in D \cap M^*$. Finally letting 
$\mathcal{M}_e = \mathcal{M}_q \cup \{M\} \cup \mathcal{M}_r \setminus M$, and
$w_e$ be the witness of the compatibility between $w_q$ and $w_r$, we have that $e= (\mathcal{M}_e, w_e) \in \mathbb{M}(\mathbb{P})$ and that 
$(e,t')$ extends both $(r,t')$ and $(q, h)$. 
\end{proof}

\begin{coro}\label{scaffoldingpreservation}
Let $T$ be a Souslin tree and let $\mathbb{P}$ be a proper forcing. Then  $\Vdash_{\mathbb{P}}$ `` $T$ is Souslin '' implies 
$\Vdash_{\mathbb{M}(\mathbb{P})}$ `` $T$ is Souslin ''. \qed
\end{coro}

\section{PFA($T$) with finite conditions}\label{PFAT}

We now show that it is possible to force an analog of the Proper Forcing Axiom 
for proper poset that preserve a given Souslin tree $T$. 
We will follow Neeman's presentation of the consistency of PFA with finite conditions, from \cite{Neeman2},
arguing that a slightly modification of his method is enough for our purposes.
Then we will argue that in the model we build $T$ remains Souslin

Recall Neeman's definition of the forcing $\mathbb{A}$ (Definition 6.1 from \cite{Neeman2}). 
Fix a supercompact cardinal $\theta$
and a Laver function $F: \theta \rightarrow H(\theta)$ as a book-keeping for choosing the 
proper posets that preserve $T$. Moreover define $Z$ as the set of ordinals $\alpha$, such 
that $(H(\alpha), F \restriction \alpha)$ is elementary in $(H(\theta), F)$. Then let 
$\mathcal{Z}^{\theta} = \mathcal{Z}_0^{\theta} \cup \mathcal{Z}_1^{\theta}$, where $\mathcal{Z}_0^{\theta}$
is the collection of all countable elementary substructure of $(H(\theta), F)$ and 
$\mathcal{Z}_1^{\theta}$ is the collection of all $H(\alpha)$, such that $\alpha \in Z$ has
uncountable cofinality - hence $H(\alpha)$ is countably closed.  Moreover, for $\alpha \in Z$, 
let $f(\alpha)$ be the least cardinal such that $F(\alpha) \in H(f(\alpha))$. Notice
that, by elementarity, $f(\alpha)$ is smaller than the next element of $Z$ above $\alpha$.

\begin{defi} 
If $\mathcal{M}$ is a set of models in $\mathcal{Z}^{\theta}$, let
$\pi_0(\mathcal{M})=\mathcal{M}\cap {\mathcal{Z}_0^{\theta}}$ and $\pi_1(\mathcal{M})=\mathcal{M}\cap \mathcal{Z}_1^{\theta}$.
\end{defi}

With an abuse of notation we will identify an $\in$-chain of models with
the set of models that belong to it. 

\begin{defi}
Let $\mathbb{M}_\theta^2$ the poset, whose conditions $\mathcal{M}_p$ are 
$\in$-chains of models in $\mathcal{Z}^{\theta}$, closed under intersection. 
If $p, q \in {\mathbb{M}_\theta^2}$, we define $p \leq q$ iff $\mathcal{M}_q \subseteq \mathcal{M}_p$.
\end{defi}

See Claim 4.1 in \cite{Neeman2} for the proof that $\mathbb{M}_\theta^2$ is 
$\mathcal{Z}^{\theta}$-strongly proper.

\begin{defi}\label{PFAposet}
Conditions in the poset $\mathbb{A}(T)$ are pairs $p = (\mathcal{M}_p, w_p)$ so that:
\begin{enumerate}
\item $\mathcal{M}_p \in \mathbb{M}_\theta^2$.
\item $w_p$ is a partial function on $\theta$, with domain contained in the (finite) set 
$\{\alpha < \theta : H(\alpha) \in p$ and $\Vdash_{\mathbb{A}(T) \cap H(\alpha)} ``F(\alpha) $ is a proper poset, that preserves $T \}$.
\item For $\alpha \in \emph{dom} (w_p)$, $w_p(\alpha) \in H(f(\alpha))$.
\item $\Vdash_{\mathbb{A}(T) \cap H(\alpha)} w_p(\alpha) \in F(\alpha)$.
\item If $M \in \pi_0(\mathcal{M}_p)$ and $\alpha \in M$, then $(p \cap H(\alpha), w_p \restriction \alpha) \Vdash_{\mathbb{A}(T) \cap H(\alpha)} ``w_p(\alpha)$
is an $(M[\dot{G}_\alpha], F(\alpha))$-generic condition'', where $\dot{G}_\alpha$ is the canonical name for a 
generic filter on $\mathbb{A}(T) \cap H(\alpha)$. 
\end{enumerate}

The ordering on $\mathbb{A}(T)$ is the following: 
$q \leq p$ iff $\mathcal{M}_p \subseteq \mathcal{M}_q$ and 
for every $\alpha \in \emph{dom}(w_p)$, 
$(\mathcal{M}_q \cap H(\alpha), w_q \restriction \alpha) \Vdash_{\mathbb{A}(T) \cap H(\alpha)} ``w_q(\alpha) \leq_{F(\alpha)} w_p(\alpha)$''.
\end{defi}

\begin{remark}\label{definition}
This inductive definition makes sense, since
$\mathbb{A}(T) \cap H(\alpha)$ is definable in any $M \in \mathcal{Z}_0^{\theta}$,
with $\alpha \in M$. 
\end{remark}

\begin{remark}\label{iff}
Condition (5) holds for $\alpha$ and $M$ iff it holds for $\alpha$ and $M \cap H(\gamma)$, whenever $\gamma \in Z \cup \{\theta\}$,
is larger than $\alpha$.
\end{remark}


\begin{defi}
Let $\beta$ be an ordinal in $Z \cup \{\theta\}$. The poset $\mathbb{A}(T)_\beta$ consists of conditions $p \in \mathbb{A}(T)$ 
such that dom$(w_p) \subseteq \beta$. 
\end{defi}

\begin{remark}
In order to simplify the notation, if $p \in \mathbb{A}(T)$, then we define $(p)_\alpha$ to be $(\mathcal{M}_p, w_p \restriction \alpha)$, 
while by $p \up H(\alpha)$ we denote $(\mathcal{M}_p \cap H(\alpha), w_p \restriction \alpha)$. Notice that 
$(p)_\alpha \in \mathbb{A}(T)_\alpha$ and $p \up H(\alpha) \in \mathbb{A}(T) \cap H(\alpha)$.
\end{remark}

Following Neeman it is possible to prove the following facts. See \cite{Neeman2} for their proofs in the 
case of the forcing $\mathbb{A}$ i.e. the poset that forces PFA with finite conditions. Indeed, the only 
difference between $\mathbb{A}$ and $\mathbb{A}(T)$
is that the Laver function $F$ picks up a smaller class of proper posets; namely the class of proper poset
that preserve $T$.

\begin{teo}\label{s-proper}(Neeman, Lemma 6.7 in \cite{Neeman2})
Let $\beta \in Z \cup \{\theta\}$. Then $\mathbb{A}(T)_\beta$ is $\mathcal{Z}^{\theta}_1$-strongly proper. \qed
\end{teo}

\vspace{0cm} 

\begin{cla}\label{extension}(Neeman, Claim 6.10 in \cite{Neeman2})
Let $p$, $q \in \mathbb{A}(T)$. Let $M \in  \pi_0(\mathcal{M}_p)$ and suppose that $q \in M$. Suppose that for some $\delta < \theta$, $p$ extends 
$(q)_\delta$ and dom$(w_q) \setminus \delta$ is disjoint from dom$(w_p)$. Suppose further that $(\mathcal{M}_p \cap M) \setminus H(\delta) \subseteq \mathcal{M}_q$. 
Then there is $w_{p'}$ extending $w_p$ so that dom($w_{p'}) = $dom$(w_p) \cup ($dom$(w_q) \setminus \delta)$ and so 
that $p' = (\mathcal{M}_p, w_{p'})$ is a condition in $\mathbb{A}(T)$ extending $q$. \qed
\end{cla}

\vspace{0cm}

\begin{teo}\label{properA}(Neeman, Lemma 6.11 in \cite{Neeman2})
Let $\beta \in Z \cup \{\theta\}$. Let $p$ be a condition in $\mathbb{A}(T)_\beta$. Let $\theta^* > \theta$ and let $M^* \prec H(\theta^*)$ 
be countable with $F, \beta \in M^*$. Let $M = M^* \cap H(\theta)$ and suppose that $M \in \pi_0(\mathcal{M}_p)$. Then:
\begin{enumerate}
\item for every $D \in M^*$ which is dense in $\mathbb{A}(T)_\beta$,there is $q \in D \cap M^*$ which is compatible with $p$. 
Moreover there is $r \in \mathbb{A}(T)_\beta$ extending both $p$ and  $q$, so that $\mathcal{M}_r \cap M \setminus H(\beta) \subseteq \mathcal{M}_q$, 
and every model in $\pi_0(\mathcal{M}_r)$ above $\beta$ and outside $M$ are either models in $\mathcal{M}_p$ or of the form $N' \cap W$, where $N'$ is a model in $\pi_0(\mathcal{M}_p)$.
\item $p$ is an $(M^*, \mathbb{A}(T)_\beta)$-generic condition. \qed
\end{enumerate}
\end{teo}

\vspace{0cm}


\begin{teo}\label{PFA(T)}(Neeman, Lemma 6.13 in \cite{Neeman2})
After forcing with $\mathbb{A}(T)$, PFA($T$) holds. \qed
\end{teo} 

In order to show that $\mathbb{A}(T)$ preserves $T$, we need the following claim.  

\begin{cla}\label{working}
If $\Vdash_{\mathbb{A}(T)_\alpha} ``T$ is Souslin'', then $\Vdash_{\mathbb{A}(T)_\alpha \cap H(\alpha)} ``T$ is Souslin''.
\end{cla}
\begin{proof}
In order to show that $\mathbb{A}(T)_\alpha \cap H(\alpha)$ preserves $T$, we use
the equivalent formulation of Claim \ref{charact}. Then, fix a countable $M^* \prec H(\theta^*)$, with $\theta^* > \theta$ and $\alpha, T \in M^*$.
Then, following Remark \ref{definition}, both $\mathbb{A}(T)_\alpha \cap H(\alpha)$ and $\mathbb{A}(T)_\alpha$ are definable in $M^*$. 
If $p \in (\mathbb{A}(T)_\alpha \cap H(\alpha)) \cap M^*$, then we want to show that there is a condition $p' \leq p$ such that
for every $t \in T_{\delta_{M^*}}$, with $\delta_{M^*} = M^* \cap \omega_1$, the condition $(p',t)$ is
$(M^*, (\mathbb{A}(T)_\alpha \cap H(\alpha)) \times T)$-generic.

Let $M = M^* \cap H(\theta)$ and $\mathcal{M}_{p^M}$ be the closure under intersection of $\mathcal{M}_p \cup \{M\}$. 
It is easy to check that it is possible to find a function $w_{p^M}$ with the same domain of $w_{p}$ 
such that $p^M = (\mathcal{M}_{p^M}, w_{p^M})$
is a condition in $\mathbb{A}(T)_\alpha$ and such that $p^M \up H(\alpha) \leq p$. 
We claim that $p^M \up H(\alpha)$ is the condition we need: i.e. 
$(p^M \up H(\alpha), t)$ is an $(M^*, (\mathbb{A}(T)_\alpha \cap H(\alpha)) \times T)$-generic condition,
for every $t \in T_{\delta_{M^*}}$. 

To this aim fix a set $D \in M^*$ dense in$(\mathbb{A}(T)_\alpha \cap H(\alpha)) \times T$,
let $t \in T_{\delta_{M^*}}$ and assume $(p^M \up H(\alpha), t) \in D$.
By Theorem \ref{properA}, $p^M$ is an $(M^*, \mathbb{A}(T)_\alpha)$-generic condition. Then, 
thanks to our hypothesis, $(p^M, t)$ is an $(M, \mathbb{A}(T)_\alpha \times T)$-generic condition. 

Now define $E$  to be the set of conditions $(q, h) \in \mathbb{A}(T)_\alpha \times T$ such that
$(q \up H(\alpha), h) \in D$ and such that $\mathcal{M}_{p^M} \cap M \subseteq \mathcal{M}_q$. Notice that $(p^M, t) \in E$ and
$E \in M^*$. The set $E$ may not be dense, but $E_0 = E_0^{\leq} \cup E_0^{\perp}$,
where
$$
E_0^{\leq} = \{(q_0, h_0) : \exists (q, h) \in E \mbox{ such that } (q_0, h_0) \leq (q, h)\},
$$
and
$$
E_0^{\perp} = \{(q_0, h_0) : \forall (q, h) \in E \mbox{ } (q_0, h_0) \perp (q, h)\},
$$
is a dense subset of $\mathbb{A}(T)_\alpha \times T$ belonging to $M^*$. 

Then there is $(q_0, h_0) \in E_0 \cap M^*$ that is compatible with $(p^M, t)$. Since $(p^M, t) \in E$,
by definition of $E_0$, there is a condition $(q, h) \in E$ that is compatible with $(p^M, t)$. 
By elementarity we can assume $(q, h) \in E \cap M^*$. Now, the key observation is that by 
strong genericity of the pure side conditions if $(r, t)$ witnesses that $(p^M, t)$ and $(q, h)$ are compatible,
then $(r \up H(\alpha), t)$ witnesses that $(p \up H(\alpha), t)$ and $(q \up H(\alpha), h)$ are compatible. 
This is sufficient for our claim, because by definition of $E$ and since $q$ is finite, 
$(q \up H(\alpha), h) \in D \cap M^*$.
\end{proof}

We can now state and proof the main preservation theorem of this section. 

\begin{teo}\label{preservation}
If $G$ is a generic filter for $\mathbb{A}(T)$, then in $V[G]$ the tree $T$ is Souslin.
\end{teo} 
\begin{proof}
We proceed by induction on $\beta$, proving that $\mathbb{A}(T)_\beta$ preserves $T$. 
If $\beta$ is the first element of $Z$, then $\mathbb{A}(T)_\beta = \mathbb{M}_{\theta}^2$. 

\begin{cla}
The forcing $\mathbb{M}_{\theta}^2$ preserves $T$.
\end{cla}
\begin{proof}
Let $M^* \prec H(\theta^*)$ be a countable model with $\theta^* > \theta$, containing $\mathbb{M}_{\theta}^2$ and $T$, and
let $\mathcal{M}_p \in \mathbb{M}_{\theta}^2$ be an $(M^*, \mathbb{M}_{\theta}^2)$-generic condition, with $M= M^* \cap H(\theta) \in \mathcal{M}_p$. Moreover, 
let $t \in T_{\delta_M}$, with $\delta_M = M \cap \omega_1$. Thanks to Lemma \ref{charact},
it is sufficient to show that $(\mathcal{M}_p, t)$ is an $(M^*, \mathbb{M}_{\theta}^2 \times T)$-generic condition. 

To this aim, let $D \in M^*$ be a dense subset of $\mathbb{M}_{\theta}^2 \times T$ and assume, by density of $D$, 
that $(\mathcal{M}_p, t) \in D$. Then define 
$$
E = \{h \in T : \exists \mathcal{M}_q \in \mathbb{M}_{\theta}^2 \mbox{ such that } (\mathcal{M}_q, h) \in D  \land  \mathcal{M}_p\cap M \subseteq \mathcal{M}_q \}.
$$
Since $\mathbb{M}_{\theta}^2, D, \mathcal{M}_p \cap M \in M^*$, we have $E \in M^*$. The set $E$ may not be dense in $T$
but 
$$
\bar{E} = \{\bar{h} \in T : \exists h \in E (\bar{h} \leq h) \lor \forall h \in E (\bar{h} \perp h)\}. 
$$
belongs to $M^*$ and it is dense in $T$. 

By $(M^*, T)$-genericity of $t$, there is an $\bar{h} \in \bar{E} \cap M$ that is compatible with $t$.
Moreover, since $(\mathcal{M}_p, t) \in D$, we have that $t \in E$. Since $t \in E$ and $\bar{h} \in \bar{E}$ are compatible, 
by definition of $\bar{E}$, there is $h \in E$, with $\bar{h} \leq h$. 
By elementarity pick such an $h$ in $M^*$.
Then, by definition of $E$, there is $\mathcal{M}_q \in \mathbb{M}_{\theta}^2$, with
$\mathcal{M}_p \cap M \subseteq \mathcal{M}_q $, such that $(\mathcal{M}_q, h) \in D$. By elementarity we can find $\mathcal{M}_q \in M^*$. Then, since $\mathcal{M}_p$ is $(M, \mathbb{M}_{\theta}^2)$-strong generic
and $\mathcal{M}_p \cap M \subseteq \mathcal{M}_q$, we have that $\mathcal{M}_p$ and $\mathcal{M}_q$ are compatible. Finally, $t$ and $\bar{h}$ are compatible because
$t \leq \bar{h}$ and $\bar{h} \leq h$. Hence $(\mathcal{M}_p, t)$ and $(\mathcal{M}_q, h)$ are compatible in $\mathbb{M}_{\theta}^2 \times T$ and
this compatibility, together with the fact that $(\mathcal{M}_q, h) \in D \cap M^*$, witnesses that $(\mathcal{M}_p, t)$ is $(M^*, \mathbb{M}_{\theta}^2 \times T)$-generic. 

\end{proof}

If $\beta$ is the successor of $\alpha$ in $Z$, then, by inductive hypothesis $\mathbb{A}(T)_\alpha$ preserves $T$.
In order to show that $\mathbb{A}(T)_\beta$ also preserves $T$, we use the characterization of Lemma \ref{charact}. 
Then, let $M^* \prec H(\theta^*)$ be a countable model, with $\theta^* > \theta$, containing $\beta, F$ and $T$. Notice that
$\mathbb{A}(T)_\beta$ is definable in $M^*$, with $\beta$ as a parameter. 
Moreover let $p \in \mathbb{A}(T)_\beta$ be an $(M^*, \mathbb{A}(T)_\beta)$-generic condition, with $M= M^* \cap H(\theta) \in \mathcal{M}_p$, 
and let $t \in T_{\delta_M}$, with $\delta_M = M \cap \omega_1$. Then we want to show that 
$(p, t)$ is an $(M^*, \mathbb{A}(T)_\beta \times T)$-generic condition. 

By elementarity of $M^*$, $\alpha \in M^*$. Now, fix a $V$-generic filter $G$ over $\mathbb{A}(T)_\alpha$, with 
$(p)_\alpha \in G$. By Theorem \ref{properA} $(p)_ \alpha$ is 
an $(M^*, \mathbb{A}(T)_\alpha)$-generic condition for $M^*$ and so $M^*[G] \cap V = M^*$.

If $H(\alpha) \notin \mathcal{M}_p$ and $p$ cannot be extended to a condition containing $H(\alpha)$, then
$\mathbb{A}(T)_\beta$, below $p$, is equivalent to $\mathbb{A}(T)_\alpha$. Then, forcing below $p$,
the conclusion follows by inductive hypothesis. Then, assume $H(\alpha) \in \mathcal{M}_p$. 

Let $G_\alpha =  G \cap H(\alpha).$
Then, by Theorem \ref{s-proper}, we have that $G_\alpha$ is a $V$-generic filter
 on $\mathbb{A}(T)  \cap H(\alpha)$, because $\mathbb{A}(T)_\alpha \cap H(\alpha) = \mathbb{A}(T)  \cap H(\alpha)$. 
Without loss of generality, 
we can assume $\Vdash_{\mathbb{A}(T) \cap H(\alpha)} ``F(\alpha)$ is a proper poset that preserves $T$'',
because, otherwise $\mathbb{A}(T)_\beta$ is equal to $\mathbb{A}(T)_\alpha$ and again the conclusion
follows by inductive hypothesis. 
Let $\mathbb{Q} = F(\alpha)[G_\alpha]$.
Then, by properness of $\mathbb{Q}$ in $V[G_\alpha]$, modulo extending $p$, we can assume $\alpha \in $dom$(w_p)$.

Fix $D \subseteq \mathbb{A}(T)_\beta \times T$ dense and in $M^*$. Without loss of generality assume $(p, t) \in D$.
Since we will work in $V[G_\alpha]$, we need to ensure that $\Vdash_{\mathbb{A}(T) \cap H(\alpha)} ``T$ is Souslin''. 
But this is true, by inductive hypothesis, as the Claim \ref{working}  shows.

Now, in $V[G_\alpha]$, define $E$ to be the set of couples $(u, h) \in \mathbb{Q} \times T$ for which there is a condition $(q, h) \in \mathbb{A}(T)_\beta \times T$ such that 
\begin{enumerate}
\item $w_q(\alpha)[G_\alpha] = u$,
\item $\mathcal{M}_p \cap M \subseteq \mathcal{M}_q$,
\item $ (q , h) \in D$, and 
\item $q \up H(\alpha) \in G_\alpha$.
\end{enumerate}

Notice that $E \in M^*[G_\alpha]$ and that $(w_p(\alpha)[G_\alpha], t) \in E$. The set $E$ may not be dense, but
if we define
$E_0 = E_0^{\leq} \cup E_0^{\perp}$, with
$$
E_0^\leq = \{(u_0, h_0) \in \mathbb{Q} \times T : \exists (u, h) \in E \mbox{ } (u_0, h_0) \leq (u, h)\}
$$ 
and
$$
E_0^\perp = \{(u_0, h_0) \in \mathbb{Q} \times T : \forall (u, h) \in E \mbox{ } (u_0, h_0) \perp (u, h)\},
$$
we have that $E_0$ is dense in $\mathbb{Q} \times T$. 
Moreover, notice that by elementarity $E_0$ is in $M^*[G_\alpha]$.

Now, since $M \in \pi_0(\mathcal{M}_p)$ and $\alpha \in M^* \cap H(\theta) = M$, we have that
$\Vdash_{\mathbb{A}(T) \cap H(\alpha)} ``w_p(\alpha)$ is an $(M^*[\dot{G}_\alpha], F(\alpha))$-generic condition'',
where $\dot{G}_\alpha$ is a $\mathbb{A}(T) \cap H(\alpha)$-name for $G_\alpha$. Moreover, 
$\Vdash_{\mathbb{A}(T) \cap H(\alpha)} ``F(\alpha)$ is a proper poset that preserves $T$'' and,
by inductive hypothesis and Lemma \ref{working}, $\Vdash_{\mathbb{A}(T) \cap H(\alpha)} ``T$ is Souslin''. 
Then by Lemma \ref{charact} applied in $V[G_\alpha]$ we have that
$(w_p(\alpha)[G_\alpha], t)$ is an $(M^*[G_\alpha], \mathbb{Q} \times T)$-generic condition.

Hence, there is a condition $(u_0, h_0) \in E_0 \cap M^*[G_\alpha]$ that is compatible with $(w_p(\alpha)[G_\alpha], t)$. 
Moreover, since $(w_p(\alpha)[G_\alpha], t) \in E$ we have that $(u_0, h_0) \in E_0^{\leq}$. This means that there is $(u, h) \in E$
such that $(u_0, h_0) \leq (u, h)$. By construction $(u, h)$ is compatible with $(w_p(\alpha)[G_\alpha], t)$ 
and by elementarity we can find such a condition in $M^*[G_\alpha]$.
Let $u_\alpha \in \mathbb{Q}$ be a witness of the compatibility between $w_p(\alpha)[G_\alpha]$ and $u$. Notice that $u_\alpha$ is an
$(N[G_\alpha], \mathbb{Q})$-generic condition for all $N \in \pi_0(\mathcal{M}_p)$, with $\alpha \in N$, because $u_\alpha \leq w_p(\alpha)[G_\alpha]$.
Since $(u, h) \in E$ there is a condition $q \in \mathbb{A}(T)_\beta$, 
with $\mathcal{M}_p \cap M \subseteq \mathcal{M}_q$ and  $w_q(\alpha)[G_\alpha] = u$,
such that $(q, h) \in D$. By elementarity let $q \in M^*[G_\alpha]$ and so $(q, h) \in  M^*[G_\alpha] \cap D$.
Since $M^*[G_\alpha] \subseteq M^*[G]$ and  $M^*[G] \cap V = M^*$, we have $(q, h) \in D \cap M^*$.
Now, by strong genericity of the pure side conditions, letting $\mathcal{M}_r$ be the closure under intersection
of $\mathcal{M}_p \cup \mathcal{M}_q$, we have that $\mathcal{M}_r$ witnesses that $\mathcal{M}_p$ and $\mathcal{M}_q$ are compatible. 
Moreover every model in $\pi_0(\mathcal{M}_r)$ above $\beta$ and outside $M$ are either models in $\mathcal{M}_p$ or of the 
form $N' \cap W$, where $N'$ is a model in $\pi_0(\mathcal{M}_p)$ and $W \in \pi_1(\mathcal{M}_q)$. 
Then $u_\alpha$ is an $(N[G_\alpha], \mathbb{Q})$-generic condition,
for all $N \in \pi_0(\mathcal{M}_r)$, with $\alpha \in N$, because of Remark \ref{iff} together with the fact that
 $u_\alpha$ extends both $w_p(\alpha)[G_\alpha]$ and $u$.

Finally, back in $V$, let $\dot{u}$ and $\dot{u}_\alpha$ be $\mathbb{A}(T)_\alpha \cap H(\alpha)$-names for
$u$ and $u_\alpha$. Moreover, let $e \in \mathbb{A}(T)_\alpha \cap H(\alpha)$ be
sufficiently strong to force all the properties we showed for $q$,  $\dot{u}$ and $\dot{u}_\alpha$.
We can also assume that $e$ extends both $q \up H(\alpha)$ and $p \up H(\alpha)$. 
Now notice that $\mathcal{M}_e \cup \mathcal{M}_r$ is already an $\in$-chain closed under intersection and so 
if $\mathcal{M}_s = \mathcal{M}_e \cup \mathcal{M}_r$ and $w_{s} = w_e \cup \{\alpha, \dot{u}_\alpha\}$, we have that
$s$ is a condition in $\mathbb{A}(T)_\beta$. 
Hence $(s, t)$ witnesses that $(p, t)$ and $(q, h)$ are compatible.

\smallskip
\smallskip
\smallskip

If $\beta$ is a limit point of $Z$, let again $M^* \prec H(\theta^*)$ be a countable model
containing $\mathbb{A}(T)_\beta$ and $F$. Then if $p \in \mathbb{A}(T)_\beta$,
with $M^* \cap H(\theta) = M \in \mathcal{M}_p$, and $t \in T_{\delta_M}$, with $\delta_M = M \cap \omega_1$,  
then, thanks to Lemma \ref{charact}, it is sufficient to show that $(p, t)$
is an $(M^*, \mathbb{A}(T)_\beta \times T)$-generic condition, in order to prove that $\mathbb{A}(T)_\beta$
preserves that $T$ is Souslin. 

To this aim, let $\bar{\beta} = sup(\beta \cap M^*)$ 
and let $\delta < \bar{\beta}$, in $Z \cap M^*$, be such that dom$(w_p) \subseteq \delta$. Moreover 
fix $D \in M^*$ dense in $\mathbb{A}(T)_\beta \times T$ and assume $(p, t) \in D$.


Now, define $E$  as the set of conditions $((q)_\delta, h) \in \mathbb{A}(T)_\delta \times T$ 
that extend  to conditions $(q, h) \in D$, 
with $\mathcal{M}_p\cap M \subseteq \mathcal{M}_q$. 
The set $E$ belongs to $M^*$, but it may not be dense in $\mathbb{A}(T)_\delta \times T$.
However the set $E_0 = E_0^{\leq} \cup E_0^{\perp}$ is dense in $\mathbb{A}(T)_\delta \times T$ and belongs to $M^*$; where
$$
E_0^{\leq} = \{( q_0, h_0) \in \mathbb{A}(T)_\delta \times T : \exists ( (q)_\delta, h) \in E \mbox{ such that } ( q_0, h_0)
\leq ( (q)_\delta, h)\},
$$ 
and 
$$
E_0^{\leq} = \{( q_0, h_0) \in \mathbb{A}(T)_\delta \times T : \forall ( (q)_\delta, h) \in E \mbox{  } ( q_0, h_0)
\perp ( (q)_\delta, h)\}.
$$

Then, by the inductive hypothesis, find a condition $(q_0, h_0) \in E_0 \cap M^*$
that is compatible with $( (p)_\delta, t)$. Moreover, since $( (p)_\delta, t) \in E$
and it is compatible with $(q_0, h_0)$, we have that $(q_0, h_0) \in E_0^{\leq}$.
Then, by definition of $E_0^{\leq}$, there is a condition $((q)_\delta, h) \in E$ such that 
$( q_0, h_0) \leq ( (q)_\delta, h)$ and, so, that is compatible with 
$( (p)_\delta, t)$. By elementarity pick such a condition in $M^*$.
Moreover, thanks the fact that $\mathcal{M}_p\cap M \subseteq \mathcal{M}_q$ and that $\mathcal{M}_p \cap M$ witnesses the $M$-strong
genericity of $\mathcal{M}_p$, we have that the compatibility between 
$( (p)_\delta, t) = ( (p)_\beta, t)$
and  $( (q)_\delta, h)$ is witnessed by a condition $\big{(}(\mathcal{M}_r, w_1) , t\big{)}$,
where $\mathcal{M}_r$ is the closure under intersection of $\mathcal{M}_p \cup \mathcal{M}_q$.
Then we have that $\mathcal{M}_r \cap M \setminus H(\beta) \subseteq \mathcal{M}_q$, and that every model in $\pi_0(\mathcal{M}_r)$ 
above $\beta$ and outside $M$ are either models in $\mathcal{M}_p$ or of the form $N' \cap W$, where $N'$ is a model in $\pi_0(\mathcal{M}_p)$
and $W \in \pi_1(\mathcal{M}_q)$.

Now, let $( q, h) \in D$ witness that $((q)_\delta, h) \in E$. 
By elementarity, we can find $( q, h) \in D\cap M^*$. Then,  thanks to the fact that $\mathcal{M}_r \cap M \setminus H(\beta) \subseteq \mathcal{M}_q$
we can apply Claim \ref{extension} and find a function
$w_{2}$, extending $w_1$, defined as dom$(w_2) = $dom$(w_1) \cup ($dom$(w_q) \setminus \delta)$, 
such that $\big{(} (\mathcal{M}_r, w_2), t\big{)}$ extends $( q, h)$.
Setting $w_{r} = w_2 \cup w_p\restriction [\bar{\beta}, \beta)$, we claim that $r$ belongs to
$\mathbb{A}(T)_\beta$. 

In order to show that this latter claim holds, it is sufficient to show that if $\alpha \in $dom$(w_p)\up[\bar{\beta}, \beta)$,
then $p \up H(\alpha)$ forces that $w_r(\alpha) = w_p(\alpha)$ is an $(N[\dot{G}_\alpha], F(\alpha))$-generic condition, 
where $\dot{G}_\alpha$ is the canonical name for a $V$-generic filter over $\mathbb{A}(T) \cap H(\alpha)$ and $N \in \pi_0(r)$, with $\alpha \in N$.
Notice that $\alpha \in N$ implies $N \notin M$. 
Then, since $p$ is a condition, the claim follows thanks to Remark \ref{iff} and the fact that every model in $\pi_0(\mathcal{M}_r)$ 
above $\beta$ and outside $M$ are either models in $\mathcal{M}_p$ or of the form $N' \cap W$, where $N'$ is a model in $\pi_0(\mathcal{M}_p)$.
 
Hence, finally we have that $(r, t)$ belongs to
$\mathbb{A}(T)_\beta \times T$ and that, by construction, it
extends both $(q, h)$ and $(p, t)$.

\end{proof}

\section{Conclusions}
As stated in the introduction, Theorem \ref{preservation} could be generalized to other forcings that 
admit a formulation similar to Lemma \ref{charact}. Indeed the argument patterns of all new results of 
this paper are similar and in proving them we did not use essential properties of a tree $T$, except the
characterization of Lemma \ref{charact}. More formally, given a proper forcing $\mathbb{P}$ we could define the 
following property for a forcing $\mathbb{Q}$.

\begin{quote}
$(*)_{[\mathbb{P},\mathbb{Q}]}$: the forcing $\mathbb{Q}$ is proper and if $M$ is a countable elementary substructure of 
$H(\theta)$, for $\theta$ sufficiently large such that $\mathbb{P}$, $\mathbb{Q} \in M$, then if 
$p$ is an $(M, \mathbb{P})$-generic condition and $q$ is an $(M, \mathbb{Q})$-generic condition, then 
$(p, q)$ is an $(M, \mathbb{P} \times \mathbb{Q})$-condition.
\end{quote}

Then Theorem \ref{preservation} shows that $(*)_{[\mathbb{A}(T)_{\alpha}, T]}$ holds, for every $\alpha \in Z \cup \{\theta\}$. 
Notice that the forcing $\mathbb{A}(T)$ is not, properly speaking, an iteration.
Hence it is not fully correct to say that the property ``$T$ is Souslin'' is preserved under finite support iteration. 
However, we think that understang the pure side conditions in
terms of a real iteration would allow to extend the class of properties 
for which these preservation results hold.


\end{document}